\documentclass[12pt]{amsart}

\usepackage[dvips]{graphicx}
	\usepackage{amsmath,amssymb,amsthm,wrapfig,amsfonts,enumerate,latexsym}

\usepackage{comment,color}

\newtheorem{theorem}{Theorem}[section]
\newtheorem{lemma}[theorem]{Lemma}
\newtheorem{corollary}[theorem]{Corollary}
\newtheorem{proposition}[theorem]{Proposition}

\newtheorem{question}[theorem]{Question}

\theoremstyle{definition}

\newtheorem{definition}[theorem]{Definition}
\newtheorem{remark}[theorem]{Remark}
\newtheorem{condition*}[theorem]{Condition}
 
\numberwithin{equation}{section}

\begin{document}

\title[${\rm CAT}(0)$ cube complexes]{Random groups have  fixed 
points \\
on {\boldmath  $\mathrm{CAT}(0)$} cube complexes}
\author
[K. Fujiwara]{Koji Fujiwara}
\author
[T. Toyoda]{Tetsu Toyoda}

\thanks{The first author is supported by 
Grant-in-Aid for Scientific Research (No. 19340013).}

\email
[Koji Fujiwara]{fujiwara@math.is.tohoku.ac.jp}
\email
[Tetsu Toyoda]{tetsu.toyoda@math.nagoya-u.ac.jp}

\address
[Koji Fujiwara]
{\endgraf GSIS, Tohoku University, 
\endgraf
Aoba-ku, Sendai, Miyagi, 980-8579, Japan}
 \address
[Tetsu Toyoda]
{\endgraf Graduate School of Mathematics, 
\endgraf
Nagoya University, Chikusa-ku, Nagoya, 464-8602, Japan}

\date{}


\begin{abstract}
We prove that a random group has fixed points when 
it isometrically acts on a ${\rm CAT}(0)$ cube complex.
We do not assume that the action is simplicial.
\end{abstract}

\maketitle

\section{Introduction}\label{into-sec}

In \cite{G}, Gromov introduced random groups of the graph model 
to show the existence of groups which cannot embed uniformly into Hilbert spaces 
(see Theorem \ref{expander-gromov}). 
He also showed that when a random group of the graph model 
isometrically acts on finite or infinite dimensional Hadamard manifolds, 
there exist common fixed points.
Silberman \cite{S} provided details of the argument for 
the case of Hilbert spaces (see Theorem \ref{silberman-th}
for the precise statement).

\begin{definition}[random group of Gromov's graph model] 
Let $G=(V,E)$ be a finite  graph such that $V$ is the set of vertices 
and $E$ is the set of edges. 
Orient the edges $E$ arbitrarily. 
Fix $k$ alphabets $s_1 ,\ldots ,s_k$. 
For each $e\in E$, choose an element $a(e)$ independently, uniformly at random from 
$\{ s_1 ,\ldots s_k, s_1^{-1},\ldots s_k^{-1}\}$. 
Let $c=e_1^{\epsilon_1} \cdots e_n^{\epsilon_n}, e_i \in E$ be a cycle in $G$,
where $\epsilon_i=1$ or $-1$, and $e^{-1}$ means the edge $e$ with the 
orientation reversed.
The cycle $c$ defines  a {\it random word} $a(c)=a(e_1)^{\epsilon_1} \cdots a(e_n)^{\epsilon_n}$
on $\{ s_1 ,\ldots s_k, s_1^{-1},\ldots s_k^{-1}\}$. 
Let $R_G$ be the set of the random words $a(c)$  for all cycles in $G$. 
In this way, we obtain 
a probability distribution over 
groups $\Gamma (G)=\langle s_1 ,\ldots ,s_k | R_G \rangle$. 
We say $\Gamma(G)$ is the {\it random group} associated to $G$ with $k$ generators. 
\end{definition}
Properties of the random group $\Gamma (G)$, for example, infinite, word-hyperbolic,
or property (T) etc, become random variables.

\begin{definition}[sequence of expanders]
A \textit{sequence of expanders} is a sequence $\{G_n =(V_n ,E_n )\}$ of finite graphs 
which satisfies the following properties: 
\begin{enumerate}
\item
The number of vertices of $G_n$ goes to infinity as $n$ goes to infinity.
\item
There exists $d$ such that $\mathrm{deg}(v)\le d$ for all $v\in V_n$ 
and all $n$, where 
$\mathrm{deg}(v)$ is the number of edges at vertex $v$. 
\item
There exists $\lambda >0$ such that $\lambda_1 (G_n)\ge\lambda$ for all $n$.  
\end{enumerate}
\end{definition}
In the above definition, $\lambda_1 (G)$ is the first positive eigenvalue of 
 the combinatorial Laplacian of $G$.
For a finite graph $G=(V,E)$, the combinatorial Laplacian $\Delta_{G}$  
acts on 
real-valued function $f$ on $V$ as 
$$
\Delta_G f (v)
=
f(v)-
\sum_{u\in V\hspace{1mm}\textrm{such that }\{v,u\}\in E}\frac{1}{\mathrm{deg}(u)}f(u), \quad 
v\in V. 
$$
Its first positive eigenvalue 
$\lambda_1 (G)$ can be computed variationally as 
$$
\lambda_1 (G)
=
\inf_{\phi}
\frac{\sum_{\{u,v\}\in E}\|\phi (u)-\phi (v)\|^2}{
\sum_{v\in V}\mathrm{deg}(v)\|\phi (v)-\overline{\phi}\|^2} ,
$$
where the infimum is taken over all nonconstant maps $\phi :V\to\mathbb{R}$, and 
$\overline{\phi}=\sum_{v\in V}\frac{\mathrm{deg}(v)}{2|E|}\phi (v)$.

The girth of a graph is the minimal 
length of a cycle in the graph. 
Recall that a group $\Gamma$ has property (T) if and only if every affine isometric action of $\Gamma$ 
on a Hilbert space has a common fixed point (see \cite[Ch.4]{HV}).@
\begin{theorem}[Silberman \cite{S}]\label{silberman-th}
If $\{G_n =(V_n ,E_n )\}$ is a sequence of expanders, $3\leq\mathrm{deg}(u)\leq d$ for all $u\in V_n$ and all $n$, 
and the girth of $G_n$ is 
large enough, 
then the group $\Gamma (G_n )$ has property (T) with high probability. 
\end{theorem}

Formally he showed that 
given $k\in\mathbb{N}$, $d\in\mathbb{N}$, $\lambda >0$, there exists an explicit 
constant $g =g(k,\lambda)$ such that 
if $\{G_n =(V_n ,E_n )\}$ is a sequence of expanders such that for all $n$, 
$\lambda\leq\lambda_1 (G_n)$, 
the girth of $G_n$ is at least $g$, and $3\leq\mathrm{deg}(u)\leq d$ 
for all $u\in V_n$, then the probability for the random group $\Gamma (G_n )$ 
generated by $k$ elements  
to have property (T) 
is at least $1-ae^{-b|V_n |}$, where $a$, $b$ are explicit and only depend on the parameters 
$k$, $d$ and $\lambda$.  
The statement in Theorem 1.5 and Theorem 2.2 should be understood
similarly.

Izeki and Nayatani \cite{IN} introduced an invariant $0\le\delta\le 1$ for 
a complete ${\rm CAT}(0)$ space. 
See Definition \ref{delta-def} for the definition of $\delta$. 
A recent theorem by Izeki, Kondo and Nayatani 
generalizes a Hilbert space (in the definition of 
property (T)) in Theorem \ref{silberman-th} to a complete ${\rm CAT}(0)$ space $Y$ with 
$\delta (Y) <1$ (Theorem \ref{IKN-fpp}).

A {\it cube complex} is a metric polyhedral complex 
in which each cell is isometric to the Euclidean cube 
$\lbrack 0,1\rbrack^n$, and the gluing maps are isometries.
In this paper, we study the geometry of ${\rm CAT}(0)$ cube complexes in comparison with 
Euclidean and more generally Hilbert spaces, using the invariant $\delta$. 
The following is our main technical result. 

\begin{theorem}\label{1/2}
Let $X$ be a complete $\mathrm{CAT}(0)$ cube complex. 
Then we have 
$$
\delta (X)\leq
\frac{1}{2}
.
$$
\end{theorem}

Therefore Theorem \ref{IKN-fpp} applies to complete ${\rm CAT}(0)$ cube complexes and 
we obtain the following.

\begin{theorem}\label{fpp-for-ccc}
If $\{G_n =(V_n ,E_n )\}$ is a sequence of expanders, $2\leq\mathrm{deg}(u)\leq d$ for all $u\in V_n$ and all $n$, 
and the girth of $G_n$ is 
large enough, 
then 
any isometric action of $\Gamma (G_n )$ on a complete 
${\rm CAT}(0)$ cube complex has a common fixed point, with high probability. 
\end{theorem}

We want to emphasize that we do not assume that 
the actions are simplicial, the cube complexes are locally compact,
nor finite dimensional.
For simplicial actions, the conclusion is not new since it follows from Theorem 1.4 
and the following theorem.

\begin{theorem}[Niblo-Reeves, \cite{NR}]\label{niblo-reeves}
If a group with property (T) 
acts on a complete ${\rm CAT}(0)$ cube complex
by simplicial isometries, then it has a common fixed point.
\end{theorem}

We note that although they assume 
that $X$ is finite dimensional in the paper, 
they are only using that $X$ is complete in the
proof. 
But the assumption that the isometries are 
simplicial is essential in their argument.

We give an example of a non-simplicial isometry without a common fixed
point on a ${\rm CAT}(0)$ cube complex,
which is not Euclidean.
Let $P_n, n \in {\mathbb Z}$ be Euclidean planes with tessellations by unit
squares.
For each $n$, choose a base vertex $v_n \in P_n$ and join $v_n$ and $v_{n+1}$
by a unit edge.
This is our cubical ${\rm CAT}(0)$ complex $Y$.
We define a simplicial isometry on $Y$, $s$, as a shift by $1$ sending
$(P_n,v_n)$ to $(P_{n+1},v_{n+1})$
preserving the square structure.
Let $r$ be an isometry of $Y$ rotating each $P_n$ around $v_n$.
Generally $r$ is not simplicial.
Now the composition $sr$ is a desired isometry.

We record another consequence of Theorem \ref{1/2}. 
\begin{definition}
Let $(X,d_X)$ and $(Y,d_Y)$ be metric spaces. 
A map $f: X\to Y$ is called a \textit{uniform embedding} if there exists 
unbounded non-decreasing functions $\rho_1 ,\rho_2 :\lbrack 0,\infty )\to\lbrack 0,\infty )$ 
such that 
$$
\rho_1 (d_X(x,x'))\le d_Y(f(x),f(x'))\le\rho_2 (d_X(x,x')), 
$$
for all $x,x'\in X$. 
For a sequence $\{(X_n ,d_{X_n}) \}$ of metric spaces, 
we call a sequence of maps $f_n :X_n \to Y$ a \textit{uniform embedding} if 
there exists 
unbounded non-decreasing functions $\rho_1 ,\rho_2 :\lbrack 0,\infty )\to\lbrack 0,\infty )$ 
such that 
$$
\rho_1 (d_{X_n}(x,x'))\le d_Y(f_n (x),f_n (x'))\le\rho_2 (d_{X_n}(x,x')), 
$$
for all $n$ and all $x,x'\in X_n$. 
\end{definition}

In \cite{G}, Gromov showed the following. 
This is a key step to produce a finitely generated group which does not uniformly 
embed in a Hilbert space. 

\begin{theorem}[Gromov, see \cite{Ro}]\label{expander-gromov}
Let $\{ G_n \}$ be a sequence of expanders.
Then there is no uniform embedding of 
$\{G_n \}$ into a Hilbert space.
\end{theorem}

Using Theorem 1.1, we can adapt the proof of Theorem \ref{expander-gromov} 
to a complete ${\rm CAT}(0)$ cube complex as follows. 
Similar variations  have been well-known to experts, 
for example, see Kondo \cite{K2} in the case of 
a complete ${\rm CAT}(0)$ space with 
$\delta <1$, and also Mendel-Naor \cite{MN}.

\begin{corollary}\label{expander-ccc}
Let $Y$ be a complete ${\rm CAT}(0)$ cube complex.
Let $\{ G_n \}$ be a sequence of expanders. 
Then there is no uniform embedding of 
$\{G_n \}$ into $Y$. 
\end{corollary}

As we will see in the proof of Theorem \ref{1/2},
we have a uniform control of the tangent
cone of a cube complex 
in comparison to 
a Hilbert space.
We ask the following question.

\begin{question}\label{question.cube}
Does a complete ${\rm CAT}(0)$ cube complex always uniformly
embed in a Hilbert space ?
\end{question}

It is known that if the complex
is finite dimensional then such embedding exists (\cite{Ni}).
In their argument the assumption of finite dimension is essential. 
Since a sequence of expanders does not uniformly embed
in a ${\rm CAT}(0)$ cube complex, it seems there is 
no easy argument to answer this question 
in the negative.

The paper is organized as follows. 
In Section \ref{delta-sec},  we recall the definition and basic 
facts about the invariant $\delta$ and quote the recent theorem by Izeki-Kondo-Nayatani
which generalizes Theorem \ref{silberman-th}.
In Section \ref{proof-sec}  we prove Theorem \ref{1/2} and Theorem \ref{fpp-for-ccc}.
In Section \ref{graph-model-sec}, we recall the proof of 
Theorem \ref{expander-gromov} and explain how to modify it for 
Corollary \ref{expander-ccc} using Theorem \ref{1/2}.   

We'd like to thank G. Yu for explaining the proof of 
Theorem \ref{expander-gromov}.

\section{Relation between distortion and $\delta$}\label{delta-sec}
We assume that the readers are familiar to the definition and elementary
facts on ${\rm CAT}(0)$ spaces.
A standard reference is the book \cite{BH}.
We recall 
the definition of the  invariant introduced by Izeki-Nayatani in \cite{IN}. 
\begin{definition}[\cite{IN}]\label{delta-def}
Let $Y$ be a complete $\mathrm{CAT} (0)$ space 
and $\mathcal{P} (Y)$ be the space of all finitely supported probability measures $\mu$ on 
$Y$ each of whose supports $\mathrm{supp}(\mu )$ contains at least two points. 
For $\mu\in\mathcal{P}(Y)$, we define 
\begin{equation*}
0 \le \delta (\mu )
=
\inf_{\phi :\mathrm{supp}(\mu )\to\mathcal{H}}\frac{\|\int_{Y}\phi (p)\mu (dp)\|^2}{\int_{Y}\|\phi (p)\|^2 \mu (dp)} \le 1, 
\end{equation*}
where the infimum is taken over all maps $\phi :\mathrm{supp}(\mu )\to\mathcal{H}$ 
with $\mathcal{H}$ a Hilbert space 
such that 
\begin{align}
\|\phi (p)\| &=d\left( p,\mathrm{bar}(\mu )\right),\label{umbrellaborn}\\
\|\phi (p)-\phi (q)\| &\leq d(p ,q)\label{1lip}
\end{align}
for all $p,q \in\mathrm{supp}(\mu )$, where 
$\mathrm{bar}(\mu )\in Y$ is the barycenter of $\mu$. 
Notice that such a map $\phi$ exists.
To see that, fix a unit vector $e \in \mathcal{H}$.
Define $\phi(p)= d(p,\mathrm{bar}(\mu))e$.
Then by the triangle inequality, \eqref{1lip} is satisfied. 
We define the \textit{Izeki-Nayatani} invariant $\delta (Y)$ of $Y$ as
\begin{equation*}
0 \le \delta (Y)=\sup_{\mu\in\mathcal{P} (Y)}\delta (\mu )
\le 1.
\end{equation*}
\end{definition}

In \cite{IN}, \cite{IKN}, \cite{IKN2} it is proved by Izeki, Kondo and Nayatani 
that certain classes of groups 
must have fixed points when they isometrically act on 
complete ${\rm CAT}(0)$ spaces $Y$ if $\delta(Y)$ is small enough. 
Among these, they proved that a random group of the graph model has the fixed point property for the 
complete ${\rm CAT}(0)$ spaces whose $\delta$ are at most some constant less than $1$.

\begin{theorem}[Izeki-Kondo-Nayatani \cite{IKN2}]\label{IKN-fpp}
Let $0\leq C<1$. 
If $\{G_n =(V_n ,E_n )\}$ is a sequence of expanders, $2\leq\mathrm{deg}(u)\leq d$ for all $u\in V_n$ and all $n$, 
and the girth of $G_n$ is 
large enough 
then with high probability, 
any isometric action of the random group $\Gamma (G_n )$ on a complete 
${\rm CAT}(0)$ space $Y$ with $\delta (Y)\leq C$ has a common fixed point. 
\end{theorem}

Generally computation of the Izeki-Nayatani invariant is difficult. 
Followings are examples of spaces for which we know some estimations: 
\begin{itemize}
\item
Assume that $Y$ is a finite or infinite dimensional Hadamard manifold
or an $\mathbb{R}$-tree. Then we have $\delta (Y ) = 0$.
\item
Assume that $Y_p$ is the building $PSL(3,\mathbb{Q}_p)/PSL(3,\mathbb{Z}_p)$. Then
$\delta (Y_p ) \geq
\frac{(\sqrt{p}-1)^2}{2(p-\sqrt{p}+1)}$.
When $p = 2$, we have $\delta (Y_2 ) \leq 0.4122\ldots$
\end{itemize}
In \cite{To2}, the second author obtained a geometric condition for $\delta$ to be bounded from above by 
a constant less than $1$. 
Kondo \cite{K2} constructed the first example of a complete ${\rm CAT}(0)$ space with $\delta=1$

Although the Izeki-Nayatani invariant is defined as a global invariant of the space,  
it can be estimated by the local property of the space. 
To see this, we define the following notation, which is introduced in 
\cite{IN}.

\begin{definition}
Let $Y$ be a complete $\mathrm{CAT}(0)$ space, 
and $O\in Y$. 
we define $\delta (Y, O)\in\lbrack 0,1\rbrack$ to be
$$
\delta (Y, O )=
\sup\left\{\delta (\nu)\hspace{1mm}|\hspace{1mm}
\nu\in\mathcal{P}(Y), \mathrm{bar}(\nu )=O 
\right\} .
$$ 
If no such $\nu$ exists, we define
$\delta(Y,O)=0$. 
\end{definition}

The following lemma is basic (see \cite[Lemma3.3]{To3}). 

\begin{lemma}\label{cone-delta}
Suppose that $Y$ is a complete $\mathrm{CAT} (0)$ space. Then we have 
\begin{equation}
\delta (Y)=
\sup\{\delta (TC_p Y,\hspace{0.8mm} O)\hspace{1mm}|\hspace{1mm}
p\in Y\}
=
\sup\{\delta (TC_p Y)\hspace{1mm}|\hspace{1mm}
p\in Y\} .
\label{cone-delta-converse}
\end{equation}
\end{lemma}

We recall the following definition from \cite{IKN2}.

\begin{definition}\label{radial-def}
Let $T$ be a metric cone with the origin $O_T$. 
Let $D_{\mathrm{rad}}(T)$ be the infimum of $D$ satisfying the following condition: 
there exists a map $f :T\to\mathcal{H}$ to a Hilbert space $\mathcal{H}$ 
such that 
\begin{equation}\label{radial-cond}
\| f(v)\| =d_T (O_T ,v) 
\end{equation}
and
$$
\frac{1}{D}d_T (v,w )\leq\| f(v)-f(w)\|\leq d_T (v,w)
$$
for all $v,w\in T$. 
This number $D_{\mathrm{rad}}(T)$ is called the \textit{radial distortion} of $T$. 
If there exists no such $f$, we define $D_{\mathrm{rad}}(T)=\infty$. 
\end{definition}

Izeki, Kondo and Nayatani \cite{IKN2} 
proved the following relation between 
this invariant and $\delta$. 

\begin{lemma}[\cite{IKN2}]\label{distortion-delta}
Let $T$ be a complete $\mathrm{CAT}(0)$ metric cone with the origin $O_T$. Then we have
$$
\delta (T, O_T)\leq 1-\frac{1}{D_{\mathrm{rad}} (T)^2} .
$$
\end{lemma}

We will estimate the radial distortion of each tangent cones of $\mathrm{CAT}(0)$ cube complexes in the 
next section.

\section{Embeddings of tangent cones of cube complexes}\label{proof-sec}

In this section, we prove Theorem \ref{1/2}. 
Let $C$ be the (metric) product of 
half lines $\ell_1, \cdots, \ell_n$ such that 
each $\ell_i=[0,\infty)$.
For each non-empty subcollection of $\{ 1,\cdots,n \}$, 
$\{i_1 < i_2 < \cdots <i_m \}$,
we define a subset, called a {\it face of dimension $m$} of $C$, 
by
$$\{ 0\} \times \cdots \times\{ 0 \} \times \ell_{i_1} \times \{ 0\}
 \times \cdots \times\{ 0\}
\times \ell_{i_m} \times \{ 0\} \times\cdots\times\{ 0\},$$
where all coordinates other than $i_j$ are $0$.
(This is like a face of a cube.)
We say the point $0\times \cdots \times 0$
is the face of dimension $0$.

\begin{lemma}[cube lemma]\label{cube}
Let $C$ be a finite product of half lines.
Let $X$ be the union of several faces of $C$.
Denote the path metric in $X$ by $d_X$, and 
let $d_C$ be the (path) metric in $C$.
Then for any points $p,q \in X$,
$$ d_X(p,q)/\sqrt{2} \le d_C(p,q) \le d_X(p,q)$$

\end{lemma}
\proof
$ d_C(p,q) \le d_X(p,q)$ is clear. So we show 
the other inequality.
There is a unique face $P$ of $C$ which has the least
dimension among faces which contain $p$.
Also, there is a unique face $Q$ for $q$.
Let $R=P\cap Q$ be the intersection, which is also a face. 
Let $p',q' \in R$ be the projection of $p,q$, respectively.
Look at the path
$$[p,p']\cup [p',q'] \cup [q',q] \subset (P \cup Q) \subset X$$
The angle at $p',q'$, measured in $C$ are $\pi/2$, 
and therefore we compute
$$d_C(p,q)^2=d_C(p,p')^2+d_C(p',q')^2+d_C(q',q)^2$$
On the other hand, to estimate $d_X(p,q)$, 
we compute $d_{P\cup Q}(p,q)$.
Let's look for the shortest 
path among the paths from $p$ to $q$ in $P \cup Q$
which starts from $p$ then reaches
some point $r \in [p',q']$, then 
goes to $q$.
(Imagine for example $P$ is a $3$-cube and $Q$
is a square such that $P$ and $Q$ share an edge, which
is $R$.)
We then find a path the  square of whose length is 
$$d_C(p',q')^2+(d_C(p,p')+d_C(q,q'))^2$$
This number gives an upper bound of $d_X(p,q)^2$.
On the other hand, this number is 
$\le 2d_C(p,q)^2$ by the above computation, so that 
$d_X(p,q)^2 \le 2 d_C(p,q)^2$.

\qed

\begin{lemma}\label{at-vertex}
Let $X$ be a ${\rm CAT}(0)$ cube complex and $p$ a vertex.
Then there exists a $1$-Lipschitz
map $\phi:TC_p(X) \to \mathcal{H}$
such that 
$$
d_{\mathcal{H}}(0,\phi (v))=d_{TC_p X} (O,v), \quad
d_{TC_pX}(u,v)/\sqrt{2} \le d_{\mathcal{H}}(\phi (u),\phi (v)) \le d_{TC_pX}(u,v)$$
for all $u,v\in TC_p X$. 
\end{lemma}
\proof
Let $e_i (i \in I)$ be the set of all edges of $X$ which contain
$p$.
For each $n$-cube $C$ in $X$ which contains $p$,
any point $x \in C$ is uniquely written 
as $x=\sum_i^n t_{i} e_{j_i}$
such that $e_{j_i}$ are the edges which are contained in $C$, 
and $1 \ge t_i \ge 0$.
In this way, any point $x \in C$ 
is uniquely written as $x=\sum_i t_i e_i$
such that $1 \ge t_i \ge 0$ and $t_i=0$ except for finitely many ones. 
Similarly, any point $x \in TC_pX$ 
is uniquely written as
$$x=\sum_i t_i e_i , $$
where $0 \le t_i \le 1$ and $t_i=0$ except for finitely many ones. 
To be precise we have abused the notation such that 
$e_i$ is now the unit vector in $TC_pX$ which 
corresponds to the direction of the edge $e_i$.
Now prepare a set of orthonormal vectors
$v_i \in \mathcal{H}, (i \in I)$.
To define a map $\phi:TC_pX \to \mathcal{H}$,
for $x=\sum_i t_i e_i$, 
set $\phi(x)=\sum_i t_i v_i$.
Now apply Lemma \ref{cube}.
\qed

\begin{remark}
In the proof of the above lemma, 
we use the
following condition 
which is satisfied 
at each vertex $v$ of a ${\rm CAT}(0)$ cube complex $X$: 
Let $e_i (i \in I)$ be the collection of all edges
which contain $v$.
For each finite subset $J \subset I$, 
there is at most one cube in $X$ which contains
all of $e_j, j \in J$ and whose dimension is 
$|J|$.

\end{remark}

Thus, we have proved the following proposition. 

\begin{proposition}\label{ccc-drad}
Let $Y$ be a $\mathrm{CAT}(0)$ cube complex, and $p\in Y$.  
Then we have 
\begin{equation*}
D (TC_p Y)\leq
\sqrt{2}
,
\quad
D_{\mathrm{rad}}(TC_p Y)
\leq
\sqrt{2}.
\end{equation*}
\end{proposition}

\begin{proof}

Notice that we may assume without loss of generality that 
$p$ is a vertex. This is because if $p$ is not 
a vertex, namely, $p$ is an interior point of an $n$-dimensional
cube, then there exists another ${\rm CAT}(0)$ cube complex $X'$ and a vertex
$p'$ 
such that a neighborhood of $p$ in $X$ is isometric to 
a neighborhood of $(p',0)$ in $X' \times {\Bbb R}^n$.
Then, $TC_pX$ is isometric
to $TC_{p'}X' \times {\Bbb R}^n$.

Now the proposition follows from Lemma \ref{at-vertex}. 
\end{proof}

Combining Proposition \ref{ccc-drad} with Lemma \ref{distortion-delta} and Lemma \ref{cone-delta},
we obtain Theorem \ref{1/2}. 
Now Theorem 2.2 applies to ${\rm CAT}(0)$ cube complexes, therefore we obtain Theorem 1.5.

\section{Wang's invariant and uniform embeddability of expanders}
\label{graph-model-sec}

In this section, 
we prove Corollary \ref{expander-ccc}. 
For a finite graph $G$ and a complete ${\rm CAT}(0)$ 
space $Y$, Wang \cite{Wa} defined the first positive
eigenvalue of a combinatorial Laplacian,
$\lambda_1(G,Y)$, for maps from the set of 
vertices of $G$ to $Y$.
This is a natural generalization of the one 
for the usual combinatorial Laplacian, $\lambda_1(G)$,
for real-valued functions on the set of vertices.

\begin{definition}[$\lambda_1(\Gamma,Y)$ by Wang]
Let $G=(V,E)$ be a finite graph, and $Y$ be a complete ${\rm CAT}(0)$ space. 
We assume that $Y$ contains at least two points. 
The Wang's invariant $\lambda_1 (G, Y)$ is defined by 
$$
\lambda_1 (G,Y)
=
\inf_{\phi}
\frac{\sum_{\{u,v\}\in E}d_Y (\phi (u),\phi (v))^2}{
\sum_{v\in V}\mathrm{deg}(v)d_Y (\phi (v),\overline{\phi})^2}, 
$$
where the infimum is taken over all nonconstant maps $\phi :V\to Y$, and 
$\overline{\phi}$ denotes the barycenter of the probability measure 
$
\sum_{v\in V}\frac{\mathrm{deg}(v)}{2|E|}\mathrm{Dirac}_{\phi (v)}
$ 
on $Y$. $\mathrm{Dirac}_{\phi (v)}$ is the Dirac measure at $\phi (v)\in Y$. 
\end{definition}

If we see ${\Bbb R}$ as a ${\rm CAT}(0)$ space, 
$\lambda_1(G)=\lambda_1(G,{\Bbb R})$ holds.
If we take a Hilbert space $\mathcal{H}$,
it is not hard to show from the definition 
$\lambda_1(G,\mathcal{H})=\lambda_1(G)$.
Originally, the invariant $\delta(Y)$ was 
introduced to give an estimate of $\lambda_1(G,Y)$. 
Izeki and Nayatani have shown the following.
\begin{theorem}[Izeki-Nayatani]\label{lambda}
\cite[Proposition 6.3]{IN}
Let $G$ be a finite graph and $Y$ a complete 
${\rm CAT}(0)$ space. Then we have
$$(1-\delta(Y))\lambda_1(G) \le \lambda_1(G,Y) \le 
\lambda_1(G).$$
\end{theorem}

By Theorem \ref{1/2} and \ref{lambda}, we immediately obtain 
the following. 
\begin{corollary}\label{cube.lambda}
Let $X$ be a complete ${\rm CAT}(0)$ cube complex,
and $G$ be a finite graph.
Then
$$ \lambda_1(G)/2 \le \lambda_1(G,X)
\le \lambda_1(G).$$
\end{corollary}

The invariant $\lambda_1(G,Y)$ is 
important not only in the study of fixed point theorems, 
but also in connection to the embedding of graphs in $Y$. 
To explain the result by Gromov that a sequence of expanders 
does not uniformly embed in the Hilbert space, 
let $\{G_n =(V_n ,E_n)\}$ be a sequence of finite 
graphs such that 
\begin{enumerate}
\item
The number of vertices of $G_n$ goes to infinity as $n$ goes to infinity.
\item
There exists $d>0$ such that $\mathrm{deg}(v)\leq d$ for all $v\in V_n$ and all $n$.
\end{enumerate}
If the sequence $G_n$ uniformly embeds into a Hilbert space 
$\mathcal{H}$, then one can show 
\begin{equation}\label{embedding}
\liminf_n \lambda_1(G_n) =0.
\end{equation}
If $\{G_n \}$ is a sequence of expanders,
then by definition, it must satisfy 
$$\liminf_n \lambda_1(G_n) =c>0 ,$$
therefore we obtain the following.

\begin{theorem}[Gromov, see \cite{Ro}]
Let $\{ G_n \}$ be a sequence of expanders.
Then there is no uniform embedding of 
$\{G_n \}$ into a Hilbert space.
\end{theorem}

By Corollary  \ref{cube.lambda}, for a complete ${\rm CAT}(0)$ cube complex $Y$, 
the sequence of expanders, $\{ G_n \}$ satisfies 
\begin{equation}\label{ccc-lambda1}
\liminf_n \lambda_1(G_n,Y)  \ge \liminf_n\lambda_1(G_n)/2
=c/2>0 .
\end{equation}
As for the equation \eqref{embedding},
the same conclusion follows from the same argument
if we replace the Hilbert space by a complete
${\rm CAT}(0)$ space $Y$. 
Similar observation can be found in  \cite{K2} and \cite{MN}
as we mentioned. For the readers' convenience, we record
the argument.

\begin{theorem}\label{lambda1-th}
Let $\{ G_n =(V_n ,E_n)\}$ be a sequence of finite graphs satisfying the conditions 
(1) and (2). 
If $\{ G_n \}$ embeds uniformly into a complete ${\rm CAT}(0)$ space $Y$, 
then we have 
$$
\liminf_{n\to\infty} \lambda_1(G_n ,Y) =0. 
$$
\end{theorem}

\begin{proof}
Let 
$\{ f_n :G_n \to Y\}$ be a uniform embedding and 
$\rho_1 ,\rho_2 :\lbrack 0,\infty )\to\lbrack 0,\infty )$ be unbounded non-decreasing functions such that 
$$
\rho_1 (d_{G_n}(x,x'))\le d_Y(f_n (x),f_n (x'))\le\rho_2 (d_{G_n}(x,x')), 
$$
for all $x,x'\in G_n$ and all $n$.  
Put $c=\rho_2 (1)$, then we have 
$d(f_n (v) ,f_n (w))\leq c$ for all $\{ v,w\}\in E_n$ and all $n$. 

First, observe that for any $r>0$, 
the preimage of an $r$-ball 
has diameter at most $\rho_1^{-1}(2r)$. 
Thus 
the preimage of any $r$-ball in $Y$ by any $f_n$ contains at most $d^{1+\rho_1^{-1}(2r)}$ vertices of $G_n$ 
since $G_n$ satisfies the property (2).

Next, 
by the definition of $\lambda_1 (G_n ,Y)$, if $\lambda_1 (G_n ,Y)>0$ we have 
$$
\sum_{v\in V_n}d_Y (f_n(v),\overline{f_n})^2
\le
\frac{1}{\lambda_1 (G_n ,Y)} \sum_{\{v,w\}\in E_n}
d_Y (f_n(v),f_n(w))^2 .
$$
The right-hand side is no greater than $\frac{1}{\lambda_1 (G_n ,Y)}\frac{d |V_n |}{2} c^2 $. 
Thus at least $\frac{|V_n |}{2}$ terms in the sum on the left-hand side 
are at most $\frac{1}{\lambda_1 (G_n ,Y)}c^2 d$. 
This means that 
the preimage of the ball of radius 
$\left(\frac{d}{\lambda_1 (G_n ,Y)}\right)^{\frac{1}{2}}c$
centered at $\overline{f_n}\in Y$ contains at least $\frac{|V_n |}{2}$ vertices of $G_n$. 

To argue by contradiction we assume that  
$$
\liminf_n \lambda_1(G_n ,Y) =\lambda >0. 
$$
Then there exists arbitrarily large $n$ such that $\lambda_1 (G_n ,Y)>\frac{\lambda}{2}$. 
Put $r=\left(\frac{2d}{\lambda}\right)^{\frac{1}{2}}c$. 
The above discussion implies that for such $n$ 
the preimage of the $r$-ball centered at $\overline{f_n}$ by $f_n$ 
contains at least $\frac{|V_n|}{2}$ vertices of $G_n$. 
But then $\frac{|V_n|}{2}\leq d^{1+\rho_1^{-1}(2r)}$, impossible. 
\end{proof}

Now, Corollary \ref{expander-ccc} follows immediately from \eqref{ccc-lambda1} and Theorem \ref{lambda1-th}.

\end{document}